\documentclass[a4paper,reqno,11pt]{amsart}
\usepackage{amsmath,amssymb,amsthm,mathrsfs}

\usepackage{fullpage}

\usepackage{color}
\definecolor{darkblue}{rgb}{0.0,0.0,0.85}
\definecolor{darkgreen}{rgb}{0,0.5,0}

\usepackage{hyperref}  
\hypersetup{colorlinks,breaklinks,            linkcolor=darkblue,urlcolor=darkblue,            anchorcolor=darkblue,citecolor=darkblue}

\begin{document} 
 
\author[B. Cheng \&\ A. Mahalov]{Bin Cheng and Alex Mahalov}
\title[Time-averages of Geophysical Fluid Dynamics]{Time-averages of Fast Oscillatory Systems in Three-dimensional Geophysical Fluid Dynamics and Electromagnetic Effects}

\newtheorem{theorem}{Theorem}
\newtheorem{lemma}[theorem]{Lemma}
\newtheorem{corollary}[theorem]{Corollary}

\theoremstyle{definition}
\newtheorem{definition}[theorem]{Definition}
\newtheorem{proposition}[theorem]{Proposition}
\newtheorem{remark}[theorem]{Remark}

\numberwithin{equation}{section}
\numberwithin{theorem}{section}

\def\mR{{\mathbb R}}
\def\mS{{\mathbb S}^2}
\def\H{{  H}}
\def\LL{{{  L}^2}}
\def\cC{{\mathscr C}}
\def\SS{\boldsymbol S}

\def\trsp{^\intercal}
\def\ssperp{{\scriptstyle{\perp}}}
\def\vu{{\mathbf u}}
\def\vv{{\mathbf v}}
\def\vb{{\mathbf b}}
\def\vce{{\vec{e}}}
\def\vuh{{\mathbf u_h}}
 \def\ur{w}
 \def\vFe{{\mathbf F}_{\textnormal{ext}}}
\def\vg{{\mathbf g}}

\def\lgt{\phi}
\def\lat{\theta}
\def\nm{\vec{n}}
\def\tg{\vec{\tau}}
\def\bc{\big|_{\partial\Omega}}
\def\Bc{\Big|_{\partial\Omega}}
\def\bcSH{\big|_{\partial\OmegaSH}}
\def\bcB{\big|_{\partial\OmB}}
\def\ep{\varepsilon}
\def\pa{\partial}
\def\va{{\mathbf a}}
\def\ver{{\mathbf e_r}}
\def\vez{{\mathbf e_z}}
\def\ve{{\mathbf e}}
\def\vet{{\mathbf e_\lat}}
\def\vep{{\mathbf e_\lgt}}
 \def\vutil{\widetilde{\vu}}
\def\tPsi{\tilde{\Psi}}
\def\sX{{\mathsf{X}}}
\def\sXh{{\mathsf{X}_h}}

\def\pt{\pa_t}
\def\pap{{\pa\over\pa\lgt}}
\def\pat{{\pa\over \pa {\lat}}}
\def\pr{{\pa\over\pa r}}
\def\prs{\pa_r }
\def\dv{\textnormal{\,div\,}}
\def\curl{\textnormal{curl\,}}
\def\grad{\nabla} 
\def\rgrad{\nabla^{\ssperp}} 
\def\nt{\nabla\!\times\!}
\def\dvh{\textnormal{\,div}_{h}}
\def\curlh{\textnormal{curl}_{h}}
\def\gradh{\grad_{h}} 
\def\rgradh{\rgrad_{h}}
\def\Deltah{\Delta_{h}}
\def\intT{ \int_{ 0}^{ T}}
\def\intO{\displaystyle\int_\Omega}
\def\intr{\int_{\Ri}^{\Ro}}
\def\intS{\int_{\mS}}

\def\id{\textnormal{id}}
\def\lsm{\stackrel{{}_<}{{}_\sim}}
\def\leC{\le C}
 \def\ccdot{\cdot}
\def\na{\nabla}
\def\cn{  \cdot\!  \nabla}
\def\pn{\nm\cn}
\def\Pzonal{{\displaystyle\operatorname*{\displaystyle\Pi}_{\scriptscriptstyle\textnormal{zonal}}}}
\def\sL{{\mathscr L}}
 \def\cLh{{\mathcal L}_h}
\def\cP{{\mathcal P}}
\def\cQ{{\mathcal Q}}
\def\cPh{\cP_h}
\def\cQh{\cQ_h}
   
\def\Ri{{1 -\thick}}
\def\Ro{{1+\thick}}
\def\thick{\delta}
\def\Projh{\textnormal{Proj}_h}

\def\va{{\bf a}}
\def\vb{{\bf b}}
\def\vgo{\vg\Big|_{r=1+\thick}}
\def\vgi{\vg\Big|_{r=1-\thick}}
\def\vgoi{\vg\Big|_{r=1\pm\thick}}

\def\II{{\mathscr{X}}}
\def\OmegaSH{\Omega}
\def\tangent{/\!/}
\def\averM{{M}_0} 
\def\MT{(1+C\mu T)M_0+C_\alpha TM_0^2 }
\def\MTold{(M_0+TM_0^2)}
\def\ttan{{\textnormal{tan}}}
  
\newcommand{\aver}[1]{{\overline{#1}}}
\newcommand{\aav}[1]{{\widetilde{#1}}}
\newcommand\FT[1]{\widehat{#1}}
\newcommand\rev[1]{#1}
\newcommand\ltwo[1]{\big\langle#1\big\rangle}
 \newcommand\Lp[2]{{L^{#1}(#2)}}

 \def\bpm{\begin{pmatrix}}
\def\epm{\end{pmatrix}}
\def\be{\begin{equation}}
\def\ee{\end{equation}}
\def\ba{\begin{aligned}}
\def\ea{\end{aligned}}
\def\wtd{\widetilde}

 \begin{abstract} Time-averages are common observables in analysis of experimental data and numerical simulations of physical systems. We will investigate,  from the angle of  partial differential equation   analysis,     some  oscillatory geophysical fluid dynamics in three dimensions: Navier-Stokes  equations    in a fast rotating, spherical shell,  and  Magnetohydrodynamics  subject to strong Coriolis and Lorentz forces. Upon averaging their oscillatory solutions in time, interesting patterns such as zonal flows can emerge. More rigorously, we will   prove that, when the restoring forces are strong enough,  time-averaged solutions stay close to the null spaces of the wave operators, whereas the solutions themselves can be arbitrarily far away from these subspaces. 


Date: October, 2014.\newline

\noindent{\scshape Bin Cheng } \newline
             Department of Mathematics\newline
University of Surrey\newline
Guildford, GU2 7XH, United Kingdom\newline\verb+b.cheng@surrey.ac.uk+\newline

\noindent{\scshape Alex Mahalov} \newline
School of Mathematical and Statistical Sciences\newline
    Arizona State University, Wexler Hall (PSA)\newline
    Tempe, Arizona\quad 85287-1804\quad USA\newline\verb+mahalov@asu.edu+
 \end{abstract}
\maketitle

 \bigskip \section{\bf Introduction} 
 In many geophysical fluid dynamical systems,    solutions exhibit fast oscillatory behaviors due to strong energy-preserving, restoring mechanisms --- a typical example being the Coriolis force in fast rotating planets and stars.  Time integration averages out the oscillatory part of the solution, which leads to emerging of interesting patterns that are relevant in a longer time scale.

 A straightforward framework is introduced in \cite{chengmahalov2012} for proving that the time-average of the solution stays close to the null space of the large, skew-self-adjoint operator in the  partial differential equation (PDE)  system.  A particular application of this  framework can be found in \cite{chengmahalov2013} for two-dimensional (2D) Euler equations on a fast rotating sphere. 

In this article, we study two  PDE systems in three-dimensional (3D) spatial domains that are important models of fast oscillatory, geophysical fluid dynamics.
 
 The first one, \eqref{NS:vec} -- \eqref{BC:Navier}, governs    viscous, barotropic fluids confined within a fast rotating, spherical shell that models the global atmospheric circulation on Earth and other planets. On the boundary, the velocity field   either   satisfies conditions in terms of shear stress or is simply fixed. We prove in Theorem \ref{thm:shell} that, with additional spatial-averaging in the radial direction, time-averages of the solution are $O(\ep)$ close to zonal flows (i.e. motions in the east-west direction). Here, $\ep$ denotes the Rossby number, a dimensionless parameter measures the ratio between typical magnitudes of inertia  and Coriolis force. This theoretical result is consistent with many numerical studies and observations.  
For a partial list of computational results, we mention \cite{galperinnakanoetal2004,nakanohasumi2005} for 3D models, \cite{vallismaltrud1993, chopolvani1996, nozawayoden1997, huanggalperinetal2001, sukorianskyetal2002, galperinsukorianskyetal2006} for 2D models, and references therein. Note that many of these computations attempt to simulate turbulent flows with sufficiently high resolutions. Zonal structures in these numerical results are either directly noticeable by naked eyes or after some time-averaging procedures.

On the other hand, we have observed  zonal flow patterns (e.g. bands and jets) on giant planets for hundreds of years, which has attracted considerable interests recently thanks to spacecraft missions and the launch of the Hubble Space Telescope (e.g. \cite{garciasanchez2001}, \cite{porcowestetal2003}). In \cite{jupiterlink2001},  the banded structure is directly observable in a composite view of the   Jovian atmosphere captured by the Cassini spacecraft. There are also observational data in the oceans on Earth showing persistent zonal flow patterns (e.g. \cite{roden1998,roden2000, maximenkobangetal2005}).

It is worth mentioning that time-averaging and the more general time-filtering are commonly used for denoising of  observational and computational data. Indeed, such post-processing is necessary for the emergence of zonal flow patterns in some of the above literature, e.g. \cite[Fig 4,6 and 9]{nakanohasumi2005}. We also point out that the zonal flow pattern would not arise in a model without the meridional variation of the Coriolis parameter. Such variation  is due to the non-flat geometry of the spatial domain, which is why  we use   the \emph{entire} spherical shell as  the   domain in this article. Most analytical work in literature adopts the $\beta$-plane approximation,   focusing on a narrow strip near a fixed latitude, which essentially is a linear approximation of the spherical case.    Fourier series then become applicable in the $\beta$-plane approximation but is not so in the whole spherical shell that is studied in this article.

Mathematical studies of deterministic and stochastic 3D rotating Navier-Stokes equations including resonances were done in \cite{BMN1997AA,BMN1999IUMJ,BMN2001IUMJ,flandolimahalov2012} with uniformly large rotation. Recently,  \cite{kuksinmaiocchi2014}     proves some interesting $\beta$-plane effects using the randomly forced quasi-geostrophic equation.

%It is also closely related to theoretical studies in e.g. [Lions, Temam and Wang]

The second PDE system \eqref{vu:vb:MHD}   governs rotating Magnetohydrodynamics (MHD) in the whole three-dimensional Euclidean space $\mR^3$ subject to two strong restoring forces: Coriolis force and Lorentz force. They induce the magneticstrophic waves, also known as rotating Alfv\'en waves (\cite{davidson2001}). We refer to \cite{mahalov2014} for ionospheric applications. We prove in Theorem \ref{thm:MHD} that time-averages of the solution vanish at order of fractional powers of $\ep$ when measured in $L^s$ norms ($s>6$). This result suggests that there is dispersion in the time-averages, although we do not impose any spatial decay on the initial data like in the classical dispersive wave theory.

The rest of this article is organized as following. The formulations and main results are introduced in Sections \ref{sec:result:1}, \ref{sec:result:2}. Then,  %Sections \ref{sec:aver} and \ref{sec:shell:thm} are dedicated to the Navier-Stokes equations in a shell. 
in Section \ref{sec:aver}, we apply the barotropic averaging \eqref{def:aver} on the 3D Navier-Stokes equations \eqref{NS:vec} and reveal the close connection to the 2D Navier-Stokes equations on a sphere. In Section \ref{sec:shell:thm}, we prove the main Theorem \ref{thm:shell} using the time-averaging tools devloped in \cite{chengmahalov2013, chengmahalov2012}. In Section \ref{sec:MHD}, we study the MHD system \eqref{vu:vb:MHD}  and prove Theorem \ref{thm:MHD} by using Sobolev-type inequalities. Finally, in Section \ref{sec:app} the Appendices, we give a geometric proof of Proposition \ref{thm:Navier} regarding the Navier boundary conditions and also prove  energy and enstrophy estimates   for the Navier-Stokes equations in a shell subject to Navier boundary conditions with $\lambda\ge0$ and $\vg\equiv{\bf0}$.
\section*{Acknowledgements}
We would like to thank Boris Galperin for stimulating and insightful discussion in the geophysical context.

AM is supported, in part, by the AFOSR, grant number: FA9550-15-1-0096.

 \bigskip \section{\bf Navier-Stokes equations in a rotating shell: formulation and main results}\label{sec:result:1}
Let $(x,y,z)$ denote the usual Cartesian coordinates and
 let spherical coordinates $(r,\lat,\lgt)$ denote the radius, {\it colatitude} (i.e. inclination from the positive half of the $z$ axis) and longitude respectively. The spatial domain is a thin, spherical shell
 \be\label{def:OmegaSH}\ba\OmegaSH&:=\bigl\{(x,y,z)\,\big|\,\sqrt{x^2+y^2+z^2}\in(\Ri,\Ro)\bigr\}\\&=\bigl\{(r,\lat,\lgt)\,\big|\,r\in(\Ri,\Ro) ,\lat\in[0,\pi],\lgt\in[0,2\pi]\bigr\}.\ea\ee
 In short, we can write $\OmegaSH=(\Ri,\Ro)\times\mS$ with $\mS$ denoting the unit sphere.
 
Let $\ver,\vet, \vep$ denote the locally orthogonal unit vectors along the increasing directions of $r,\theta,\phi $ respectively --- and they are orientated according to the right hand rule, i.e. $(\ver\times\vet)\cdot\vep=1$. Similarly define $\ve_x,\ve_y, \vez$ in terms of the Cartesian coordinate system. The unknown is velocity field  $\vu$.   The   Coriolis force  is given by
\[{\mathbf F}_\text{Coriolis}={1\over\ep} \vu\times\ver \cos\lat={z\over\ep} \vu\times\ver\]
where $\ep$, called the Rossby number, equals the ratio of the spatial domain's rotating period over the inertial time scale (usually $0.01\sim0.1$ for the Earth). Note we have adopted  such geophysical version of the Coriolis force that differs from the laboratory version, ${1\over\ep}\vu\times\vez$. In other words,  we neglect the radial component of the velocity and also neglect the radial component of the Coriolis force. See \cite{whitehoskinsetc2005} for detailed justification.

Let $q$ denote the pressure and constant $\mu$   the viscosity. The sum of other   external forces is denoted by $\vFe$. 
Then,   the incompressible Navier-Stokes equations  under the Coriolis force reads (\cite{arnoldkhesin1998, chorinmarsden1993, pedlosky1987})
\be\label{NS:vec}\left\{\ba\pt\vu+\nabla_\vu\vu+\grad q&={1\over\ep} \vu\times\ver \cos\lat +\mu\Delta\vu+\vFe,\\ \dv\vu&= 0,\ea\right.\ee
subject to the Navier boundary conditions which consist two parts,
\begin{subequations} \label{BC:Navier} \begin{align} \label{BC:zero:f} \vu\ccdot\nm\bc &=0,&\text{(zero-flux)}\\\label{BC:shear}\Big[\SS  \nm+\lambda\vu \Big]_\ttan\bigg|_{\partial\Omega}&=\vg,&\text{(given shear stress)}\end{align}\end{subequations} 
with scalar $\lambda=\lambda(t,{\bf x})\ge0$ and vector $\vg =\vg (t,{\bf x})$    given. (The physical significance of having positive $\lambda$ is shown in Proposition \ref{prop:energy:decrease} and its proof.) Here, $\nm$ denotes the outward normal at $\pa\OmegaSH$ and subscript ``$\ttan$'' indicates the tangential component, e.g. $\vv_\ttan:=\vv-(\vv\ccdot\nm)\nm$. The stress tensor $\SS$ is defined as
\begin{gather*}\SS  := \nabla\vu+(\nabla\vu)\trsp\\ \text{where}
\quad\nabla\vu:=\bpm\pa_{x }u_1&\pa_{y}u_1&\pa_{z}u_1\\
\pa_{x }u_2&\pa_{y}u_2&\pa_{z}u_2\\\pa_{x }u_3&\pa_{y}u_3&\pa_{z}u_3\epm.\end{gather*}
Throughout this article, vectors are treated as $3\times1$ matrices so that for vector fields $\vu,\vu',\vu''$,
\begin{subequations}\label{nabla:vu:ids}\be\label{nabla:vu:id}(\nabla\vu)\vu'=\vu'\!\cn\vu,\ee
\be\label{nabla:trsp:vu:id}\big((\nabla\vu)\trsp\vu'\big)\ccdot\vu''=\big(\vu''\!\cn\vu\big)\ccdot\vu'.\ee\end{subequations}

One can also impose the Dirichlet boundary condition, $\vu \bc =\vg'$
with $\vg'$ given. This apparantly includes the non-slip boundary condition $\vu \bc=0$.

Before stating the main result,  some definitions are in order. First, the Sobolev $L^2$ norm for a scalar or vector function $f$ is defined as 
\[\|f\|_{L^2(\Omega)}:=\sqrt{\intO|f|^2 }.\] Second, we use the so-called ``baratropic averaging'' to reduce the 3D velocity field $\vu$ to a 2D field $\aver{\vu}$ that is tangent to $\mS$. It turns out that certain weight in the integral is convenient. From a physical perspective, the flux of $\aver{\vu}$ going through a side of the area element $\sin\theta\, d\theta d\phi$ equals the (unweighted) radial average of  the momentum flux through the corresponding vertical  cross-section  of the volume element $ r^2\sin\theta\, d\theta d\phi dr$. To this end, for velocity field $\vu\in L^2(\OmegaSH)$, define its horizontal component $$\vuh:=\vu-(\vu\ccdot\ver)\ver$$ and define its baratropic averaging
\be\label{def:aver} \aver{\vu} :={1\over2\thick}\int_\Ri^\Ro r\vuh  \,dr.\ee
%
% On a related note,  geostrophic turbulence theory such as the pioneering work of [Charney, 1971] predicts an inverse cascade of  energy  to larger scales, the rates of which are comparable in the horizontal and \emph{stretched} vertical directions (hence the so-called equi-partition). However, over the entire globe, the horizontal scale is much larger than the Rossby radius of deformation  in atmosphere-ocean dynamics, which makes such inverse cascade more efficient in the vertical direction. In other words, upon proper rescaling and subtraction of a background state, the dynamics ``looks'' more homogeneous vertically, and the term barotropization has been used in this context.
 
Next, define $\Pzonal:L^2(\mS)\mapsto L^2(\mS)$ as the zonal-mean projector that projects horizontal velocity fields onto the subspace of zonal flows,
\[\Pzonal\aver{\vu}(\lat,\lgt):=\left({1\over2\pi}\int_0^{2\pi}\aver{\vu}(\lat,\lgt)\ccdot\vep\,d\lgt\right)\vep.\]

 Lastly,   let $C$ denote some universal constant and we add subscript(s) to it,
 e.g. $C_k$, to emphasize its dependence on another parameter.

 We now state the main result subject to homogenous boundary condition $\vg=0$. We will skip the case of nonhomogenous boundary condition $\vg\ne0$ because of the intimate connection between $\vg$ and the external forcing $\vFe$ that is discussed  in Subsection \ref{subs:physical} where such connection is explained also  in a more physical context. 
 
\begin{theorem}[Homogeneous boundary conditions]\label{thm:shell} Consider 3D Navier-Stokes equations \eqref{NS:vec} in a spherical shell $\OmegaSH$ defined in \eqref{def:OmegaSH},  subject to the Navier boundary conditions \eqref{BC:Navier} with $\vg \equiv{\bf0}$. Let $\mu<1/2$ and $\thick<1/2$. Define %$\averM$ as the normalized $L^2$ norm of the initial data $\vu_0$, i.e.,
\[\averM:= {1\over\sqrt{2\thick}}{\|\vu_0\|_{L^2(\Omega)}},\]
that indicates the averaged size of $\vu_0$.
 
Then, for any weak solution $\vu\in L^\infty([0,\infty);L^2(\Omega))\cap L^2([0,\infty);H^1(\Omega))$, its barotropic average $\aver{\vu}$ as defined in \eqref{def:aver} satisfies, for   ${\alpha}<-4$,
\be\label{zonal:est}\ba&\left\|(1-\Pzonal)\intT\aver{\vu} \,dt\right\|_{H^{{\alpha}}(\mS)}\\&\le \,\ep\,   \big[\MT+ M_\text{ext}\big],\ea\ee
where $M_{\textnormal{ext}}:=\left\|(1-\Pzonal)\intT \aver{ \vFe}\right\|_{H^{{\alpha+2}}(\mS)}$ and the $H^{\alpha}(\mS)$ norm  can be defined using spherical harmonics (c.f. Definition \ref{def:Hk:series} and relation \eqref{dual:Ha}. Note for negative $\alpha$, the $H^{\alpha}$ norm dampens high wave number modes). The constant $C_\alpha$ depends solely on $\alpha$ and is otherwise independent of $\ep$, $\thick$, $\mu$, $M_0$, $T$.\end{theorem}

Since operator $(1-\Pzonal)$ effectively extracts the non-zonal component of a velocity field, estimate \eqref{zonal:est} confirms that $\intT\aver{\vu}$ is $O(\ep)$ close to zonal flows.

Combining this theorem with the energy estimate \eqref{grad:u:est}, we can obtain via interpolation that, for $\alpha \in[-4,1)$,
\[\begin{aligned}&\Big\|(1-\Pzonal)\intT\aver{\vu} \,dt\Big\|_{H^{\alpha}(\mS)}\\&\quad\le \big[\MT+ M_\text{ext}\big]\ep^a\mu^{-b},\end{aligned}\]
where $a,b$ are positive numbers depending on $\alpha$.

 We remark that the possible negative $\alpha$ values used in the $H^\alpha$ estimates above suggest that zonal flow patterns are associated with smaller wave numbers i.e. larger spatial scales, since the high wave number modes are damped in the definition of $H^\alpha$ norms for negative $\alpha$.
 
The above results in 3D are nontrivial extension from the 2D case studied in \cite{chengmahalov2013} which is centered around the Euler equations on a fast rotating unit sphere $\mS$, \be\label{Euler:sphere}\pt\vu+\nabla_\vu\vu+\grad q={z\over\ep}\vu^\perp,\qquad \dv\vu=0 ,\ee
where ${}^\perp$ denotes the $\pi/2$ counterclockwise rotation of the associated vector on $\mS$.

For comparison, the main theorem for the 2D system  \eqref{Euler:sphere} is stated as  following with some minor notational changes.
\begin{theorem}[\cite{chengmahalov2013}]\label{thm:sphere}Consider the incompressible Euler equation \eqref{Euler:sphere}   on $\mS$ with initial data  $\vu_0\in H^k(\mS)$ for $k\ge3$.  Let $M_0:=\|\vu_0\|_{H^k}$. Then, there exists     a function $f(\cdot):[-1,1]\mapsto\mR$ depending on $\vu$, so that 
\be\label{estimate:thm}\begin{aligned}\Big\|\intT\big(\vu-\rgrad f(z)&\big)\,dt\Big\|_{H^{k-3}(\mS)}\\& \le C  \MTold  \ep\,,\end{aligned}\ee
for any given $T\in[0,T^*/M_0]$ where
 constant  $  T^*$ depends on $k$ but is independent of $\ep$ and $\vu_0$.
 
In spherical coordinates, %the approximation $\rgrad f(z)$ is
\[\rgrad f(z)=-f'(\sin\theta)\sin\theta\,\vep,\]
which represents   {longitude-independent zonal flows}.
 \end{theorem}
 
 We finally remark that analysis of 3D Navier-Stokes equations and its variations in the geophysical context, including the existence of  solutions and the low Rossby number limit,  has seen substantial progress in recent years, e.g. \cite{lionstemamwang1992,lionstemamwang1995,lionstemamwang1997,BMN1997AA,BMN1999IUMJ,BMN2001IUMJ,caotiti2012,flandolimahalov2012}, just to name a few. There are also results regarding Navier-Stokes equations on thin 3D domains, e.g. \cite{temamziane1997,iftimiegaugelsell2007}, without the Coriolis effect.  The boundary conditions in the literature are either periodic, whole space,   non-slip   $\vu\bc=0$ or some variations of the Navier type \eqref{BC:Navier}. In the next two subsections, we  further discuss \eqref{BC:Navier} and its   variations as seen in literature.

 \subsection{Geometry of the Navier boundary conditions}
  
 The following proposition is regarding a general domain $\Omega$. 
% \begin{corollary}[Non-homogeneous boundary conditions]\end{corollary}
 
 \begin{proposition}\label{thm:Navier}For general smooth domain $\Omega\subset\mR^3$, let $\nm$ be the outward normal at a point of $\pa\Omega $. Suppose $$\vu\ccdot\nm\bc=0.$$ Then, the $\SS  \nm$ term in the Navier boundary conditions \eqref{BC:Navier} satisfies 
\be\label{Navier:shapeO}\begin{split}  [\SS   \nm]_\ttan\Bc &=\big[{\nm\cn\vu}\big ]_\ttan- \vu\cn\nm  \\
&=  (\curl\vu)\times\nm-2\,\vu\cn\nm  ,\end{split}\ee
where the $\vu\cn\nm$ is well defined (intrinsically) on $\pa\Omega$ due to $\vu\ccdot\nm\bc=0$.
We can further rewrite it using 
\be\label{shapeO:kappa}- \vu\cn\nm  =\sum_{i=1 }^2(\vu\ccdot\vce_i)\kappa_i\vce_i,\ee
 with $\vce_1,\vce_2$ being a pair of orthonormal, principal directions of the surface $\pa\Omega$ and $\kappa_i =-(\vce_i\cn\nm)\cdot\vce_i$   being the corresponding principal curvature. 

As a consequence, for the case of spherical shell domain defined in \eqref{def:OmegaSH}, 
\begin{subequations}\label{Navier:kappa:SH}
\begin{align}\label{Navier:kappa:SH:a}   [\SS \nm  ]_\ttan\Big|_{r=1\pm\thick} &=\pm\Big({\pa\vuh\over\pa r} - { \vuh\over  r}\Big)\\
\label{Navier:kappa:SH:b}&=  (\curl\vu)\times\nm  \mp { 2\vuh\over  r} .\end{align}
\end{subequations}
%Note here  the $\mp$ is set to be minus at the outer boundary and to be  plus at the inner boundary.
\end{proposition}
The proof is postponed to the Appendices. Also,
consult \cite{watanabe2003} for more details.

In literature, the Navier boundary conditions are also referred to as ``stree-free'' or ``slip'' boundary condition. It should be however distinguished from the so-called ``free'' boundary condition (which is confusingly referred to as ``slip'' boundary condition in some cases),\[\vu\ccdot\nm\bc=0,\qquad(\curl\vu)\times\nm\bc=0.\]The Navier boundary conditions should also be distinguished from the Neumann type boundary condition,\[\vu\ccdot\nm\bc=0,\qquad {\pa\vuh\over\pa r} \bc=0\] as used in e.g. \cite{lionstemamwang1995}.  
By \eqref{Navier:kappa:SH} of Proposition \ref{thm:Navier}, they only correspond to special cases of  the Navier boundary conditions \eqref{BC:Navier} when one lets $\vg \equiv{\bf0}$ and makes specific choices for the parameter $\lambda$. We are not aware of any physical explanation for these choices of $\lambda$. A more serious issue is that   such specific choices always involve $\lambda<0$   on the inner boundary $r=1-\thick$. However,   by   Proposition \ref{prop:energy:decrease} and its proof, the constraint $\lambda\ge0$ on the boundary {\it everywhere} is   necessary for  the dissipation of energy.  Also  it is physically invalid to argue that the above two  boundary conditions are  the small-curvature approximations of  the Navier boundary conditions \eqref{BC:Navier}, because the  principal curvatures of $\pa\Omega$ are of $O(1)$  in  global circulation models for which the radius of Earth is rescaled to near the unit.

Interested reader can further
consult \cite{iftimiegaugelsell2007}, in particular the top part of page 1085, and references   therein.

  \subsection{Physical considerations of external forcing and non-homogeneous boundary conditions}\label{subs:physical}
 
In the main Theorem \ref{thm:shell}, the external force $\vFe$ affects the estimate only via $(1-\Pzonal)\intT \aver{ \vFe}$ which is its non-zonal component averaged in time and $r$. %This can be relatively small on a planet for which $\vFe$ represents radiation. 
This external force is intimately connected to
non-homogeneous boundary conditions which are studied in e.g. the context of
planetary boundary layer (PBL). Mathematically speaking,  if  $\wtd\vu $ satisfies the Navier-Stokes equations \eqref{NS:vec} with  nonhomogeneous boundary conditions
\[ \left\{\begin{aligned}   \wtd \vu \ccdot\nm\bc &=0, \\ \Big[\SS_{\wtd\vu}  \nm +\lambda\wtd\vu  \Big]_\ttan\bigg|_{\partial\Omega}&=\vg \ne{\bf0}, \end{aligned}\right.\]
and if one can  find {\it some} velocity field $\vv$, regardless of the dynamics, that is only subject to the   boundary conditions 
\be\label{BC:vv} \left\{\begin{aligned}   \vv\ccdot\nm\bc &=0, \\ \Big[\SS_\vv  \nm+\lambda\vv \Big]_\ttan\bigg|_{\partial\Omega}&=\vg , \end{aligned}\right.\ee
Then, the new unknown $\vu:=\wtd\vu -\vv$ will satisfy   \eqref{NS:vec}  with homogeneous boundary conditions 
\[ \left\{\begin{aligned}   \vu\ccdot\nm\bc &=0, \\ \Big[\SS_\vu  \nm+\lambda\vu\Big]_\ttan\bigg|_{\partial\Omega}&={\bf0},\end{aligned}\right.\]which is then covered by the main Theorem \ref{thm:shell}. The new external force term in the $\vu$ system apparently contains information of the original boundary data $\vg$. 
%Then, this new formulation in terms of $\vu$ will be 

There are indeed infinitely many ways to construct $\vv$ satisfying \eqref{BC:vv}. For example, it suffices to find vector fields $\va(\lat,\lgt),\vb(\lat,\lgt)$ that are both tangent to $\mS$ so that $$\vv= r^2\va+\vb$$ satisfies \eqref{BC:vv}.  The $\vv\ccdot\nm\bc =0$ part is apparently valid. For the second condition of \eqref{BC:vv}, we rewrite it using   \eqref{Navier:kappa:SH:a},
 \[ \Big(\pm\Big({\pa\vv_h\over\pa r} - { \vv_h\over  r}\Big)+\lambda\vv_h\Big)\Big|_{r=1\pm\thick}=\vgoi.\]
 Substitute $\vv= r^2\va+\vb$ and rearrange
  \[\left\{\begin{aligned} \big(1+\thick+(1+\thick)^2\lambda\big)\va+\big({-1\over1+\thick}+\lambda\big)\vb&=\vgo,\\ \big(-1+\thick+ (1-\thick)^2\lambda\big )\va+\big({ 1\over1-\thick} +\lambda\big )\vb&=\vgi.\end{aligned}\right.\]
 With $\lambda\ge0$ and $\thick\in(0,1/2)$, we always have  $$ \big(1+\thick+(1+\thick)^2\lambda\big)> \big|-1+\thick+ (1-\thick)^2\lambda\big|$$ and $$\big({ 1\over1-\thick} +\lambda\big )>\big|{-1\over1+\thick}+\lambda\big|.$$ Therefore, the coefficient matrix of the above linear system is diagonally dominant. Thus, we can  perform Gaussian elimination (while pretending $\va,\vb$ to be scalar unknowns) and express $\va,\vb$ as linear combinations of $\vgo,\vgi$ which are both tangent to $\mS$.

 \bigskip \section{\bf Magnetohydrodynamics    in $\mR^3$: formulation and main results}\label{sec:result:2}
Consider the domain to be $\mR^3$ in which a uniform, imposed magnetic field $\vez$ resides and a fast rotating (about $\vez$), conducting fluid moves subject to the predominantly large Coriolis force and Lorentz force. The fluid is homogeneous, incompressible and un-magnetizable. Then, upon some scaling arguments, one can reduce the full Navier-Stokes and Maxwell's equations to the following MHD  system  \cite[\S 3.8]{davidson2001}   for the unknowns:  velocity field $\vu$  and        induced magnetic field  $\vb$ (so that   the total magnetic field is given by $\vez+\ep\vb$),
\begin{subequations}\label{vu:vb:MHD}
\be\label{vu:MHD}\left\{\begin{aligned}\pt\vu+\vu\cn \vu&+\grad q\\= \,&{  \vu\times\vez\over\ep} +{ ( \curl\vb)\times(\vez+\ep\vb)\over\ep} ,\\\dv\vu=\,&0;\end{aligned}\right.\ee
\be\label{vb:MHD} \pt\vb= \dfrac{\curl\big[\vu\times(\vez+\ep\vb)\big]}{\ep} ,\quad\dv\vb=0.\ee\end{subequations}

Here, $\ep$ denotes the MHD Rossby number as well as the ratio of the induced magnetic field over imposed magnetic field; $q$ denotes the pressure. Note  in  \eqref{vu:MHD} the Coriolis force and Lorentz force are of the same scale which is $O(1/\ep)$ times the inertia. This is a reasonable scaling since the ratio of these two forces  is often  close to 1 in many geophysical and astrophysical applications  (\cite{davidson2001}). For simplicity, we have set both the kinetic viscosity and  magnetic viscosity to be zero.

Let $M_0:=\|(\vu_0,\vb_0)\|_{H^k(\mR^3)}$ for $k>5/2$. By the standard energy method, we know 
\be\label{energy:MHD}\ba\|(\vu,\vb)\|_{H^k(\mR^3)}&\leC_k M_0 \\\text{ for positive times }\;\; &t\leC_k/ M_0.\ea\ee

%We now state  the following theorem regarding time-averages.
\begin{theorem}\label{thm:MHD}
Consider any classical solution to \eqref{vu:vb:MHD} satisfying \eqref{energy:MHD}. Then, for any positive $T\leC_k/ M_0 $,
\be\label{est:vu}\Big\|\intT\vu\,dt\Big\|_{W^{k-3,\infty} (\mR^3)}\leC_k M_0\left[\big(T+M_0T^2 \big)\ep\right]^{1\over2}  ,\ee
and
\be\label{est:vb}\Big\|\intT\vb\,dt\Big\|_{W^{k-4,s} (\mR^3)}\leC_{k,s} M_0\left[\big(T+M_0T^2 \big)\ep\right]^{{1\over6}-{1\over s}} ,\ee
with  $6<s<\infty$.
\end{theorem}
Therefore, time-averages of the solution vanish at order of fractional powers of $\ep$ when measured in $L^s$ norms ($s>6$). This result suggests there is dispersion in the time-averages, although we do not impose any spatial decay on the initial data like in the classical dispersive wave theory.

 \bigskip \section{\bf Barotropic averaging of the Navier-Stokes equations}  \label{sec:aver}

Recall  the definition of barotropic averaging \eqref{def:aver},
\be\label{def:aver:1} \aver{\vu} :={1\over2\thick}\int_\Ri^\Ro r\vuh  \,dr, \ee
where $$\vuh:=\vu-(\vu\ccdot\ver)\,\ver\,.$$
Also define the barotropic  average  for a scalar $f$,
\be\label{def:aver:f}\aver{f}:={1\over2\thick}\int_\Ri^\Ro f\,dr.\ee
 
We first remove the pressure term $\nabla q$ in  \eqref{NS:vec} using the Helmholtz-Leray decomposition. Define $\sX$ to be the space of incompressible velocity fields subject to zero-flux boundary condition, \[\ba\sX&:= \LL\text{ closure of }\\&\quad\left\{\vu^\text{inc}\in \cC^1(\overline{\Omega})\,\Big|\,\dv\vu^\text{inc}=0,\;\vu^\text{inc}\ccdot\nm\bc=0\right\}\ea\]
By using   testing   functions, we see that 
\be\label{sX:testing}\begin{aligned}\sX&= \LL\text{ closure of }\\&\quad\left\{\vu\in \cC^1(\overline{\Omega})\,\Big|\,\intO \vu\cn f=0\text{ for any }f\in \H^1(\Omega)\right\}\\
&=\left\{\vu\in \LL( {\Omega})\,\Big|\,\intO \vu\cn f=0\text{ for any }f\in \H^1(\Omega)\right\}\end{aligned}\ee
 
Define $\cP$ as the $\LL$-orthogonal projection onto  $\sX $ so that, for any $\vu,\vu'\in \LL(\Omega)$, \begin{subequations}\label{def:cP} \begin{align}\label{def:cP:1} \cP^2\vu=\cP\vu \in \sX,&\\
\label{def:cP:2}\intO(\vu-\cP\vu)\ccdot(\cP\vu')&=0.\end{align}\end{subequations}
In fact, $\cP$ is the classical Leray projection subject to zero-flux boundary condition. Then, define 
\[\cQ:=I-\cP.\]

  Now pick any scalar $f\in \H^1(\Omega)$. By orthogonality of $\cP$, $\cQ$ in \eqref{def:cP:2}, we have
  \[\intO\cP(\nabla f)\ccdot \cP(\nabla f)=\intO\cP(\nabla f)\cn f\]
 which is zero due to $\cP(\nabla f)\in\sX$ satisfying \eqref{sX:testing}. In other words,
 \[\cP(\nabla f)\equiv0.\]
By this property, we apply $\cP$ on the first equation of \eqref{NS:vec}, cancel the $\grad q$ term and  reformulate it into,
\be\label{Navier-Stokes:cL}\pt\vu+\cP(\nabla_\vu\vu)={1\over\ep}\cP( \vu\times\ver \cos\lat)+\mu\cP\Delta\vu.\ee

Note that, for generic div-free velocity field $\vu$ satisfying the  Navier boundary conditions   \eqref{BC:Navier}, the   term $\Delta\vu$ is no longer subject to the zero-flux boundary condition $\vu\ccdot\nm\bc=0$  whereas the image of $\cP$ always satisfies the zero-flux boundary condition. Thus, $\cP\Delta\vu$ and $\Delta\vu$ differ by a div-free, potential flow --- the gradient of the so-called Stokes pressure.

\subsection{\bf Barotropic averaging of Helmholtz-Leray projection}
Similar to \eqref{sX:testing}, we define
\[\sXh=\left\{\vuh\in \LL( {\mS})\,\Big|\,\intO \vuh\cn g=0\text{ for any }g\in \H^1(\mS)\right\}\]
and then define projections $\cPh$ and $\cQh:=I-\cPh$ for ``horizontal'' velocity field $\vuh\in L^2(\mS)$  so that, analogous to \eqref{def:cP} 
\begin{subequations}\label{def:cPh} \begin{align}\label{def:cPh:1}  \cPh^2\vuh =\cPh\vuh& \in \sXh,\\\label{def:cPh:2}
 \int_{\mS} \cQh\vuh \ccdot(\cPh\vuh')&=0.\end{align}\end{subequations}

Here and below,  subscript $h$ following an operator   indicates the operator acts on scalar or vector fields defined on $\mS$. For example, $\Deltah$ denotes the Laplace–Beltrami operator on $\mS$. Their properties are discussed in the appendices of \cite{chengmahalov2012}.

Now, we give the relation between $\cP$ and $\cPh$.
\begin{lemma}\label{lm:aver:Leray}
For any vector field $\vu=\cP\vu+\cQ\vu\in L^2(\OmegaSH)$,
\[\aver{\cP\vu}=\cPh\aver{\vu},\qquad\aver{ \cQ\vu}=\cQh\aver{\vu}.\]
\end{lemma}

In the proof,  we will repeatedly use the following basic facts that relate  the differential operators in $\Omega$ to those in $\mS$. \\
For vector $\vu=\ur\ver+\vuh$,
\begin{subequations}\label{3D:2D}
\be\label{3D:2D:a}\ba \dv\vu&=r^{-2}{\pa\over\pa r}(r^2 \ur )+r^{-1}\dvh\vuh,\\
\curl \vu& =  r^{-1}(\curlh\vuh)\ver\\&\quad+r^{-1}\big(\gradh \ur-\pr(r\vuh)\big)\times\ver;\ea\ee 
 for scalar $f$, \be\label{3D:2D:b}\ba \grad f=\,& \pr f\ver+r^{-1}\gradh f,\\
\Delta f=\,& r^{-2}\pr(r^2\pr f)+r^{-2}\Deltah f.\ea\ee\end{subequations}
Note that the relation of $\curl$ and $\curlh$ in polar coordinates (with $\lat$ being the colatitude) is due to  the following formulations that   roughly resemble the Cartesian-coordinate form,
\[\begin{split}\text{for }\vu=\,&\ur\ver+u_\theta\vet+u_\phi\vep,\\\curl\vu =\,&{1\over r\sin\lat}\left(\pat(u_\lgt\sin\lat)-\pap u_\lat\right)\ver\\
  &+{1\over r}\left({1\over\sin\lat}\pap \ur-\pr(ru_\lgt)\right)\vet\\
  &+{1\over r}\left(\pr(r u_\lat)-\pat(\ur)\right)\vep\,,\\
\curlh\vuh=\,&{1\over \sin\lat}\left(\pat(u_\lgt\sin\lat)-\pap u_\lat\right)\,.\end{split}\] 

\begin{proof}[Proof of Lemma \ref{lm:aver:Leray}]  
Apply barotropic averaging \eqref{def:aver:1} to  $\vu=\cP\vu+\cQ\vu$ and get 
\(\aver{\vu}=\aver{\cP\vu}+\aver{\cQ\vu}.\)
Since by definition we also have $\aver{\vu}=\cPh\aver{\vu}+ \cQh\aver{\vu}$, it suffices to prove
\be\label{cQ:aver}\aver{ \cQ\vu}= \cQh\aver{\vu}.\ee
Also, since $\H^1(\Omega)$ is dense in $\LL(\Omega)$ and barotropic averaging is apparently bounded from $\LL(\Omega)$ to $\LL(\mS)$, we only consider $\vu\in \H^1 (\Omega)$ so that $\vu\ccdot\nm$ is defined on $\pa\Omega$.

By   elliptic PDE theory,
 we have  \be\label{def:cQ}
{\cQ\vu=\nabla f} \text{ where  } \left\{\begin{split}\Delta f&=\dv\vu&\text{in }\Omega \\\nabla f\ccdot\nm&=\vu\ccdot\nm&\text{in } {\pa\Omega}\end{split}\right.\ee
Similar equations hold for $\cQh$ (c.f. \cite[(2.1), (2.2)]{chengmahalov2012}),
\be\label{def:cQh}
{\cQh\vuh=\gradh\Deltah^{-1}\dvh\vuh} \ee
Here, $\Deltah^{-1}$ is defined using spherical harmonics, and maps between  scalar functions of zero mean --- note $\int_{\mS}\dvh\vuh=0$ by Stokes' lemma and $\pa\mS=\emptyset$.

Let $\vu=w\ver+\vuh$. Use \eqref{3D:2D} to reformulate \eqref{def:cQ} as,
\be\label{Poisson:3D:2D} \left\{\ba  r^{-2}&\pr(r^2\pr  f )+r^{-2}\Deltah  f   \\ = & \,r^{-2}\pr(r^2 \ur)+r^{-1}\dvh\vuh,&\text{ in }\OmegaSH,\\
  &\quad\pr  f  =\ur,&\text{ on }\pa\OmegaSH.\ea\right.
\ee
Multiply the first equation with $r^2$ and integrate it in $r$,
\[r^2\pr f \Big|_\Ri^\Ro+\int_\Ri^\Ro\Deltah  f \,dr= r^2\ur\Big|_\Ri^\Ro+\int_\Ri^\Ro r\dvh \vuh\,dr.\]
Then, apply the second equation of \eqref{Poisson:3D:2D}  to cancel out the boundary terms,
\[ \int_\Ri^\Ro\Deltah  f \,dr=  \int_\Ri^\Ro r\dvh \vuh\,dr.\]
Since $\vu\in \H^1(\Omega)$, we can exchange integrals and derivatives above, and invoke definitions of barotropic averaging in \eqref{def:aver:1}, \eqref{def:aver:f} to obtain
 \be\label{Poisson:S2}\Deltah\aver{ f }=\dvh\aver{\vu},\quad\text{i.e.,}\quad \gradh\aver{ f }=\cQh\aver{\vu},\ee
where $\cQh$ follows   \eqref{def:cQh}.

On the other hand, apply barotropic averaging \eqref{def:aver:1} on   the first equation of \eqref{3D:2D:b} with the same $f$ as in \eqref{def:cQ} to obtain
\[\aver{ \cQ\vu}=\aver{ \grad f }={1\over2\thick}\int_\Ri^\Ro\gradh  f \,dr= \gradh\aver{ f }.\]
 Combine it with \eqref{Poisson:S2} on $\mS$, we  prove Lemma \ref{lm:aver:Leray}.
 \end{proof}

\subsection{\bf Dynamics of barotropic averages on $\mS$.}

We now apply barotropic averaging \eqref{def:aver} on the 3D Navier-Stokes equations \eqref{Navier-Stokes:cL} with the help of Lemma \ref{lm:aver:Leray} and identities \eqref{3D:2D}.
\begin{lemma}\label{lemma:Navier-Stokes:average}
The solution to \eqref{Navier-Stokes:cL}  subject to the Navier boundary conditions \eqref{BC:Navier} with $\vg \equiv{\bf0}$  satisfies
\be\label{Navier-Stokes:average}\pt\aver{\vu}+\cPh\aver{\grad_\vu\vu}={1\over\ep}\cPh(\aver{\vu}\times\ver\cos\lat) +\mu\cPh\aver{\Delta\vu}+\aver{\vFe} \ee
subject to $\dvh\aver{\vu}=0$. Here, $\vu=w\ver+\vuh$ so that $\aver{\vu}=\aver{\vuh}$.

Furthermore,  the viscosity term from above  equals  
 \be\label{aver:Delta}\cPh\aver{\Delta\vu}=\cPh\Deltah\aver{r^{-2}\vuh} +2 \cPh\aver{r^{-1}\prs\vu}.\ee
\end{lemma}

\begin{proof}

First, integrate $r^2\dv\vu=0$ in $r$ and invoke the first identity of \eqref{3D:2D:a}
\[r^2w\Big|_\Ri^\Ro+\int_\Ri^\Ro r\dvh(\vuh)=0.\]
By the zero-flux boundary condition $\vu\ccdot\nm\bc=0$, the first term vanishes, and therefore
we prove the incompressibility condition \be\label{inc:proven}\dvh\aver{\vu}=0.\ee

For the Coriolis term, $  {\cP( \vu\times\ver \cos\lat)}$, Lemma \ref{lm:aver:Leray} implies
\[ \aver{\cP( \vu\times\ver \cos\lat)}=\cPh(\aver{\vu\times\ver \cos\lat})=\cPh(\aver{\vu}\times\ver \cos\lat).\]
 
Then, upon barotropic averaging and invoking Lemma \ref{lm:aver:Leray}, the 3D Navier-Stokes \eqref{Navier-Stokes:cL} is transformed into \eqref{Navier-Stokes:average} subject to
 $\dvh\aver{\vu}=0$.

Now we show \eqref{aver:Delta}. 
By $\dv\vu=0$, we have,
\be\label{curlcurl} \aver{\Delta\vu}=-\aver{\curl\curl\vu}.\ee
For the RHS, first apply the second identity of \eqref{3D:2D:a} to get
\[\ba-\aver{\curl\vu}&=-{1\over2\thick}\int_\Ri^\Ro\big(\gradh\ur-\pr(r\vuh)\big)\times\ver\,dr\\
&=-{1\over2\thick}\left(\int_\Ri^\Ro\gradh {\ur}\,dr-r\vuh\Big|_\Ri^\Ro\right)\times\ver\ea.\]
Then,   substitute $\vu$ by $\curl\vu$  and correspondingly substitute $w=(\curl\vu)\ccdot\ver=r^{-1}\curlh\vuh$,
\[\begin{aligned}&\quad-\aver{\curl\curl\vu}\\&=-{1\over2\thick}\left(\int_\Ri^\Ro\gradh {r^{-1}\curlh\vuh}\,dr- r\Projh(\curl\vu) \Big|_\Ri^\Ro\right)\times\ver\\&= \ver\times\gradh\curlh\aver{r^{-2}\vuh} +{r\over2\thick}\Projh(\curl\vu) \Big|_\Ri^\Ro\times\ver\\&= :I+II.\end{aligned}\]

Thus, we transform the viscous term in \eqref{Navier-Stokes:average} into \be\label{cPh:I:II}\cPh\aver{\Delta\vu}=-\cPh\aver{\curl\curl\vu}=\cPh I+\cPh II.\ee

For the $\cPh I$ term, apply $\cPh$ on the identity $$\Deltah\vuh=\ver\times\gradh\curlh\vuh+\gradh\dvh\vuh$$ and then use $\cPh\gradh\equiv0$ to rewrite
\be\label{cPh:I}\cPh I=\cPh\Deltah\aver{r^{-2}\vuh}\,.\ee

For the $\cPh II$ term, invoke the second identity of \eqref{3D:2D:a} and the zero-flux boundary condition $\vu\ccdot\nm\bc=0$ to obtain
\[\Projh(\curl\vu)\Big|_\Ri^\Ro= -r^{-1}\pr(r\vuh)\times\ver\Big|_\Ri^\Ro\]
so that,
\[II= {1\over2\thick}\pr(r\vuh)\Big|_\Ri^\Ro.\]
By \eqref{Navier:kappa:SH},   the Navier boundary conditions \eqref{BC:Navier} with $\vg \equiv{\bf0}$ imply
$\pr\vuh\bc=r^{-1}\vuh$.   Therefore,
\[II= {1\over\thick}\vuh\Big|_\Ri^\Ro= 2\aver{r^{-1}\pr\vu}.\]
Combine this with \eqref{cPh:I:II}, \eqref{cPh:I}  to prove \eqref{aver:Delta}.\end{proof}

 \bigskip \section{\bf Proof of   main theorem for Navier-Stokes equations}\label{sec:shell:thm}
In this section, we following the framework in \cite{chengmahalov2012} to prove Theorem \ref{thm:shell}.

First, define\be\label{def:cLh}\cLh\aver{\vu}:=\cPh(\aver{\vu}\times\ver\cos\lat)\ee
and rewrite \eqref{Navier-Stokes:average} as
\[ \ba\cLh\aver{\vu}& =\ep\Big[\pt\aver{\vu}+\cPh\aver{\grad_\vu\vu}+\\&-2\mu\cPh\aver{r^{-1}\prs\vu}-\mu\cPh\Deltah\aver{r^{-2}\vuh}-\aver{\vFe}\Big] .\ea\]
Then, take the time-averages of each term and exchange time integration and $\cLh$,
\be\label{cLh:vw} \cLh \intT \aver{ \vu} =\ep\left[{\aver{\vu}(T)-\aver{\vu}_0 } + A_1+A_2+A_3-\aver{\vFe}\, \right] \ee
where
\[ \ba A_1:=\,& \intT\cPh\aver{\grad_\vu\vu}\,dt
\\A_2:=\,& -2\mu\intT\cPh\aver{r^{-1}\prs\vu}\,dt
\\A_3:=\,&-\mu\intT\cPh\Deltah\aver{r^{-2}\vuh}\,dt\ea\]

We will then estimate every term in the RHS of \eqref{cLh:vw} in terms of $H^k(\mS)$ norms. Note that
there are many equivalent   definitions of  Sobolev norms on a manifold  through the literature  (e.g. \cite{taylor2011}), all of which are independent of coordinate systems.  One definiation of $\|{f}\|_{H^k(\mS)}$ is
\(\sqrt{\int_{\mS}\Big[f\sum_{j=0}^{k}(- \Delta)^jf\Big]}\) for $k\ge0$. Then, by the Poincare's inequality, this definition is equivalent to \(\sqrt{\int_{\mS}\Big[f^2+f(- \Delta)^kf\Big]}\). In this article, all relavent scalar fields are of zero-mean, so that we adopt   the following definition
\be\label{def:fHK}\ba&\text{for scalar $f$ with }\int_{\mS}f=0,\\&\text{define }\|{f}\|_{H^k(\mS)}:=\sqrt{\int_{\mS}\Big[f (- \Delta)^k f\Big]}\ea\ee 
with integer $k\ge0$. 
 
Consequently, for a vector field $\vu$ on $\mS$ with Hodge Decomposition \[\vu=\gradh\Phi+\rgradh \Psi\mbox{ \quad with } \int_{\mS}\Phi=\int_{\mS}\Psi=0,\]we define its $H^{k}$ norm, among other equivalent versions, as  
\be\label{def:vuHk:general}\|\vu\|_{H^k(\mS)}:=\sqrt{\|\Phi\|^2_{H^{k+1}(\mS)}+\|\Psi\|^2_{H^{k+1}(\mS)}}.\ee
Note that, here and below, we always impose zero-mean on $\Phi$ and $\Psi$.

With the help of spherical harmonics, we extend \eqref{def:fHK} and \eqref{def:vuHk:general} to $H^\alpha$ for any $\alpha\in\mR$.
\begin{definition}\label{def:Hk:series}
Let $\{Y_l^m\}$ for $l=0,1,\ldots$ and $m=-l,\ldots,-1,0,1,\ldots,l$ be the set of spherical harmonics forming an orthonormal basis of $L^2$ such that\[\Deltah Y_l^m=-l(l+1)Y_l^m.\] Let $\vu=\rgrad\Psi$ be any div-free velocity field in $L^2(\mS)$ with
\[\Psi=\sum_{l=1}^\infty\sum_{m=-l}^l \psi_{l}^{m}Y_l^m,\quad\mbox{ where }\quad \psi_{l}^{m}=\langle{\Psi}, Y_l^m\rangle_{L^2(\mS)}.\]
Then, for any real number $\alpha$,
\[\label{P:id} \|\vu\|_{H^\alpha(\mS)}=\|\Psi\|_{H^{\alpha+1}(\mS)}:=\sqrt{\sum_{l=1}^\infty\sum_{m=-l}^l   (l^2+l)^{\alpha+1}\big|\psi_l^m\big|^2}.\]
\end{definition}

This definition allows us to easily adapt the proof of Theorem 4.1 of \cite{chengmahalov2013} and reach the next lemma (whose proof is skipped).

\begin{lemma}\label{lm:sphere}  Let $\alpha\in\mR$.
 For any   horizontal vector field $\vuh$ on $\mS$ subject to $\dvh\vuh=0$, 
\[\|(1-\Pzonal)\vuh\|_{H^{\alpha}(\mS)}\le \| \cLh\vuh \|_{H^{\alpha+2}(\mS)}.\] 
\end{lemma}
Note that $(1-\Pzonal)$ effectively extracts the non-zonal component of a velocity field.

%This duality will be used in the following proof.
\begin{proof}[Proof of Theorem \ref{thm:shell}]
 Let $\Omega$ stand for the three-dimensional shell domain defined in \eqref{def:OmegaSH} for the rest of the proof. We will also use without references the integrating-by-parts formulas on $\mS$ which can be found in e.g. \cite[(A.19)-(A.22)]{chengmahalov2012}.
 
Under Definition \ref{def:Hk:series}, it is straightforward to verify that $H^\alpha(\mS)$ and $H^{-\alpha}(\mS)$ are dual spaces with respect to the $L^2(\mS)$ inner product, namely,
\be\label{dual:Ha}\|\vu\|_{H^{-\alpha}(\mS)}=\max_{\vu'\ne0}\dfrac{\langle\vu,\vu'\rangle_{L^2(\mS)}}{\|\vu\|_{H^\alpha(\mS)}}.\ee
Then, by Lemma \ref{lm:sphere} and \eqref{inc:proven}, it suffices to estimate \[ \Big\langle  \cLh\intT\aver{ \vu},\vuh'\Big \rangle_{L^2(\mS)}\]
for  smooth, testing  vector field $\vuh'$ that is tangent to $\mS$. Since the definition \eqref{def:cLh} implies $\dvh\cLh=0$, we can further impose $\dvh{\vuh'}=0$ so that for any $\vuh''$ tangent to $\mS$,
\be\label{vu12}\big\langle  \cPh\vuh'',\vuh'\big \rangle_{L^2(\mS)}=\big\langle  \vuh'',  \vuh'\big \rangle_{L^2(\mS)}.\ee

By   \eqref{cLh:vw}, it suffices to make the following estimates. (Recall definition $\averM:=\|\vu_0\|_{L^2(\Omega)}/\sqrt{2\thick}$.)

$\bullet$ Estimate of $\aver{\vu}{(T)}$ and $\aver{\vu}_0$. 
\[\ba&2\thick\,\Big\langle\aver{\vu}{(T )},\vuh'\Big\rangle_{L^2(\mS)}
\\= &\int_\Ri^\Ro\int_{\mS}r{\vu(T )}\ccdot\vuh'
\\=\,&\Big\langle{\vu(T )},r^{-1}\vuh'\Big\rangle_{L^2(\Omega)}
\\\le\,& \|\vu(T)\|_{L^2(\Omega)}\|r^{-1}\vuh' \|_{L^2(\Omega)}
\\=\,& \|\vu(T)\|_{L^2(\Omega)}\| \vuh' \|_{L^2(\mS)}\sqrt{2\thick}
\ea\]
Since $\|\vu(T )\|_{L^2(\Omega)}\le \|\vu_0\|_{L^2(\Omega)}=\sqrt{2\thick}\averM$ by Proposition \ref{prop:energy:decrease},  we obtain
\[\|\aver{\vu}(T )\|_{L^{2}(\mS)}\le {\averM},\quad\text{and similarly }\quad \|\aver{\vu}_0\|_{L^{2}(\mS)}\le \averM \]

$\bullet$ Estimate of $A_1$. By \eqref{vu12}
\[\ba&2\thick\,\Big\langle\cPh\aver{\grad_\vu\vu},\vuh'\Big\rangle_{L^2(\mS)}
\\= &\int_\Ri^\Ro\int_{\mS}r\big[\nabla\ccdot(\vu\otimes\vu)\big]\ccdot\vuh' 
\\=\,&\Big\langle\nabla\ccdot(\vu\otimes\vu),r^{-1}\vuh'\Big\rangle_{L^2(\Omega)}
\\=\,&\Big\langle \vu\otimes\vu ,\nabla(r^{-1}\vuh')\Big\rangle_{L^2(\Omega)}\qquad\ldots\text{ by }\vu\ccdot\nm\bc=0
\\\le\,&\|\vu\|^2_{L^2(\Omega)}|\nabla(r^{-1}\vuh')|_{L^\infty(\Omega)}
\\\le \, &C\|\vu\|^2_{L^2(\Omega)}|\vuh'|_{W^{1,\infty}(\mS)}
\\\le\,&C_\beta\|\vu\|^2_{L^2(\Omega)}\|\vuh'\|_{H^{2+{\beta} }(\mS)}\qquad\text{for }\beta>0\\&\qquad\qquad \ldots\text{by Sobolev imbedding.}\ea\]
Since $\|\vu(T )\|_{L^2(\Omega)}\le \|\vu_0\|_{L^2(\Omega)}=\sqrt{2\thick}\averM$ by Proposition \ref{prop:energy:decrease}, we use the duality relation \eqref{dual:Ha} to obtain
\[ \|\cPh\aver{\grad_\vu\vu}\|_{H^{-2-{\beta}}(\mS)}\le {C_\beta{\averM}^2 },\]
and integrating in time gives\[\|A_1\|_{H^{-2-{\beta}}(\mS)}\le C_\beta T{\averM}^2 \]
for $\beta>0$.

$\bullet$ Estimate of $A_2$. By \eqref{vu12}
\[\ba&2\thick\,\Big\langle\cPh\aver{r^{-1}\prs\vu},\vuh'\Big\rangle_{L^2(\mS)}\\
=\,&\int_\Ri^\Ro\int_{\mS}\prs\vu\ccdot\vuh'
\\=\,&\Big\langle\prs\vu,r^{-2}\vuh'\Big\rangle_{L^2(\Omega)}
\\\le\,&\|\nabla\vu\|_{L^2(\Omega)} \|r^{-2}\vuh'\|_{L^2(\Omega)}
\\\le\,&C\sqrt{2\thick}\|\nabla\vu\|_{L^2(\Omega)}\|\vuh'\|_{L^2(\mS)}.\ea\]
Then, integrate in time to get,
\[\ba&\Big\langle A_2,\vuh'\Big\rangle_{L^2(\mS)}\\\le\,&  {\mu\over\sqrt{2\thick}}\,\|\vuh'\|_{L^2(\mS)}\,\int_0^T\|\nabla\vu\|_{L^2(\Omega)}
\\\le\,& {\mu\over\sqrt{2\thick}}\,\|\vuh'\|_{L^2(\mS)}\,\sqrt{T\int_0^T\|\nabla\vu\|^2_{L^2(\Omega)} }.\ea\]
Since by the enstrophy estimate \eqref{grad:u:est}, 
$$\ba\sqrt{\int_0^T \|\nabla\vu\|^2_{L^2(\Omega)}}&\le\sqrt{ \|\vu_0\|^2\left(1+{C\mu T }\right)/\mu}\\
&=\sqrt{ 2\thick\averM^2\left(1+{C\mu T }\right)/\mu}\, ,\ea$$ 
we obain
\[\|A_2\|_{L^2(\mS)}\le  C\sqrt{\mu T(1+{C\mu T})} \,\averM\]

$\bullet$ Estimate of $A_3$.  By \eqref{vu12} and Green's identity on $\mS$ (c.f. \cite[(A.22)]{chengmahalov2012})
\[\ba&2\thick\,\Big\langle\cPh\Deltah\aver{r^{-2}\vuh},\vuh'\Big\rangle_{L^2(\mS)} 
\\=\,& 2\thick\,\Big\langle\aver{r^{-2}\vuh},\Deltah\vuh'\Big\rangle_{L^2(\mS)} 
\\=\,&\int_\Ri^\Ro\int_{\mS}r^{-1}\vuh\ccdot\Deltah\vuh'
\\=\,&\Big\langle r^{-2}\vuh,r^{-1}\Deltah\vuh'\Big\rangle_{L^2(\Omega)}
\\\le\,&\|r^{-2}\vuh\|_{L^2(\Omega)}\|r^{-1}\Deltah\vuh' \|_{L^2(\Omega)}
\\\le\,&C\|\vu\|_{L^2(\Omega)}\|\Deltah\vuh' \|_{L^2(\mS)}\sqrt{2\thick}.\ea\]
Since $\|\vu(T )\|_{L^2(\Omega)}\le \|\vu_0\|_{L^2(\Omega)}=\sqrt{2\thick}\averM$ by Proposition \ref{prop:energy:decrease}, we obtain
\[\|\cPh\Deltah\aver{ r^{-2}\vu}\|_{H^{-2}(\mS)}\le C{M_0 } , \; \text{ i.e. }\; \|A_3\|_{H^{-2}(\mS)}\le C{\mu T M_0 }\]
\end{proof}

 \bigskip \section{\bf Fast Rotating MHD model}    \label{sec:MHD}
 
Let $\cP$ once again denote the Leray-Helmholtz projection. In other words, for any vector field $\vu\in L^2(\mR^3)$,
\[\vu=\cP\vu+\grad\cQ\vu\quad\]where\be\label{def:cPcQ}\cP\vu=-\curl\Delta^{-1}\curl\vu,\quad\cQ\vu=\Delta^{-1}\dv\vu.\ee
By Fourier transform, we also have
\be\label{def:cPcQ:xi}\FT{\cP\vu}(\xi)=-\dfrac{i\xi\times(i\xi\times\FT{\vu})}{|i\xi|^2},\quad \FT{\cQ\vu}(\xi)=\dfrac{i\xi\ccdot\FT{\vu}}{|i\xi|^2}.\ee
Define a skew-self-adjoint operator acting on the velocity field $\vu$ and magnetic field $\vb$,
\be\label{def:sL}\sL\bpm\vu\\\vb\epm:=\bpm\cP(\vu\times\vez+(\curl\vb)\times\vez)\\\curl(\vu\times\vez)\epm.\ee
Then, the system \eqref{vu:vb:MHD}   can be reformulated as
\be\label{sL:MHD}\pt\bpm\vu\\\vb\epm+\bpm\cP(\vu\cn\vu-\vb\cn\vb)\\\vu\cn\vb-\vb\cn\vu\epm={1\over\ep}\sL\bpm\vu\\\vb\epm\ee
where we used identity $2(\curl\vb)\times\vb-2\vb\cn\vb=\grad|\vb|^2$ to transform \eqref{vu:MHD} and identity $\curl(\vu\times\vb)=\vu(\dv\vb)-\vb(\dv\vu)+\vb\cn\vu-\vu\cn\vb$ to transform \eqref{vb:MHD}.

\subsection{Kernel of the large operator $\sL$}
It is an elementary calculation to verify that
\[\dv\vu=0\implies\curl(\vu\times\vez)=\pa_z\vu.\]
Combine it with \eqref{def:cPcQ} and the fact $\dv\vu=\dv\curl\vb=0$ to transform \eqref{def:sL} into
\[ \sL\bpm\vu\\\vb\epm=\bpm-\curl\Delta^{-1}\pa_z(\vu+\curl\vb)\\\pa_z\vu\epm.\]
Thus, by $\vb=\cP\vb=-\curl\Delta^{-1}\curl\vb$, this implies
\be\label{sL}\sL\bpm\vu\\\vb\epm= \bpm-\curl\Delta^{-1}\pa_z\vu+\pa_z\vb\\\pa_z\vu\epm\ee
Therefore,\[(\vu,\vb)\in\ker\sL\iff\vu=\vu(x,y)\text{ and }\vb=\vb(x,y)\]
When restricted to $L^2$ space, $\ker\sL=\{\bf{(0,0)}\}$.

\subsection{Control $L^\infty$ norm using $\pa_z$ derivatives}Given $H^k$ initial data ($k>5/2$),
  we can apply the standard energy method to obtain $O(1)$ estimates for the $L^p$ norms ($2\le p\le\infty$) of the solution and its first spatial derivatives   in a finite time interval. Upon time-averaging, such $O(1)$ estimates give rise to $O(\ep)$ estimates for the $L^p$ norms ($2\le p\le\infty$) of $\intT\sL\bpm\vu\\\vb\epm$. By \eqref{sL}, the $O(\ep)$ estimates also work for $\intT(\pa_z \vu,\pa_z\vb)$ in terms of $L^p$ norms\footnote{For $\pa_z\vb$, the range of $p$ is reduced to $6\le p\le\infty$ due to the negative derivative in the first line of \eqref{sL}}.

What estimates can be obtained for $\intT(\vu,\vb)$, the time-average of the solution itself?  Because of the special role of $\pa_z$ derivatives, we state and prove the following inequality regarding function $f$ defined in $\mR^1$.
\be\label{Nash}\|f\|_\Lp{\infty}{\mR^1}^2 \le C\|f\|_\Lp{2}{\mR^1} \|f'\|_\Lp{2}{\mR^1} \ee
Note that  once   the $L^\infty$ norm is estimated,   standard interpolation techniques can help control the rest of the $L^p$ norms ($2<p<\infty$).
 \begin{proof}[Proof of \eqref{Nash}] 

For any positive $\rho$, we estimate the $L^1$ norms of $\FT{f}(\xi)$ over frequencies lower and higher than $\rho$ respectively. The Holder's inequality is applied in both cases.
\[\ba\int_{-\rho}^\rho|\FT{f}(\xi)|\,d\xi&\le\Big[\int_{-\rho}^\rho|\FT{f}(\xi)|^2\Big]^{1\over2}\Big[\int_{-\rho}^\rho1\Big]^{1\over2}\\&\le C\|f\|_\Lp{2}{\mR^1}\rho^{1\over2}.\\
\int_{|\xi|>\rho}|\FT{f}(\xi)|\,d\xi&\le\Big[\int_{|\xi|>\rho}|\xi\FT{f}(\xi)|^2\Big]^{1\over2}\Big[\int_{|\xi|>\rho}\xi^{-2}\Big]^{1\over2}\\&\le C\|f'\|_\Lp{2}{\mR^1}\rho^{-{1\over2}}.\ea\]
Therefore,
\[\ba\|f\|_\Lp{\infty}{\mR^1}&\le\int_{\mR^1}|\FT{f}(\xi)|\,d\xi\\&\le C\|f\|_\Lp{2}{\mR^1}\rho^{1\over2}+C\|f'\|_\Lp{2}{\mR^1}\rho^{-{1\over2}}.\ea\]
Optimizing the RHS over $\rho\in(0,\infty)$, we prove \eqref{Nash}.
\end{proof}
We are ready to state and prove the following lemma
\begin{lemma}\label{est:R2}Given function $g(x,y,z)$, 
\[\|g\|_{L_{z}^\infty( H^m_{xy}(\mR^2))}\le C  {\|g\|_{H^m(\mR^3)}^{1\over2}\|\pa_zg\|_{H^m(\mR^3)}^{1\over2}}.\]
Consequently,  for  and $m'\ge0$ and $k>1$, \[ \|g\|_{W^{m',\infty}(\mR^3)}\le C  \|g\|_{H^{m'+k}(\mR^3)}^{1\over2}\|\pa_zg\|_{H^{m'+k}(\mR^3)}^{1\over2} .\]
 \end{lemma}
\begin{proof}
Consider $\pa_x^\alpha\pa_y^\beta g(x,y,z)$ with $\alpha+\beta\le m$. For  any numbers $a<b$, we estimate 
\[\ba&{1\over b-a}\int_a^b\|\pa_x^\alpha\pa_y^\beta g\|^2_{L^2_{xy}}dz\\=\,&{1\over b-a}\int_a^b\int_{\mR^2}\big[\pa_x^\alpha\pa_y^\beta g(x,y,z)\big]^2dxdydz\\
\le\,&\int_{\mR^2}\big\|\pa_x^\alpha\pa_y^\beta g(x,y,z)\big\|^2_{L^\infty_z(\mR^1)}dxdy\\
\le\,& C\int_{\mR^2}\Big[\int_{\mR^1}\big(\pa_x^\alpha\pa_y^\beta g\big)^2dz\Big]^{1\over2}\ccdot\Big[\int_{\mR^1}\big(\pa_z\pa_x^\alpha\pa_y^\beta g\big)^2dz\Big]^{1\over2}dxdy\\&\qquad\qquad\text{\ldots by \eqref{Nash}}\\
\le\,&C\Big[\int_{\mR^3} \big(\pa_x^\alpha\pa_y^\beta g\big)^2dzdxdy\Big]^{1\over2}\ccdot\Big[\int_{\mR^3} \big(\pa_z\pa_x^\alpha\pa_y^\beta g\big)^2dzdxdy\Big]^{1\over2} \\&\qquad\qquad\text{\ldots by H\"{o}lder's inequality in }\mR^2\\
\le\,&C \|g\|_{H^m(\mR^3)}\|\pa_zg\|_{H^m(\mR^3)}\ea\]
Because $a,b$ are arbitary, this implies $$\|\pa_x^\alpha\pa_y^\beta g\|^2_{L^\infty_z (L^2_{xy})}\le C \|g\|_{H^m(\mR^3)}\|\pa_zg\|_{H^m(\mR^3)}.$$
Summing up over all derivatives with $0\le\alpha+\beta\le m$, we complete the proof of the first inequality and the second one follows from the Sobolev inequalities.  
 \end{proof}

\subsection{Estimates on time-averages of $(\vu,\vb)$}

\begin{proof}[Proof of Theorem \ref{thm:MHD}]By taking  $\intT$ on \eqref{sL:MHD} and then taking the $H^{k-1}(\mR^3)$ norms, we have
the estimate
\be\label{est:sL}\left\|\intT\sL\bpm\vu\\\vb\epm\right\|_{H^{k-1}(\mR^3)}\le CM_0\left({1\over T}+M_0 \right)\ep.\ee
Using the second component of \eqref{sL}, we obtain from above that
\[\|\intT\pa_z\vu\|_{H^{k-1}(\mR^3)}\le CM_0\left({1\over T}+M_0 \right)\ep,\]
and together with  Lemma \ref{est:R2}, we prove \eqref{est:vu}

Similarly, substracting the second component of \eqref{sL} from the $\curl$ of the first component of \eqref{sL}, we obtain from \eqref{est:sL} that 
\[\|\intT\pa_z\curl\vb\|_{H^{k-2}(\mR^3)}\le CM_0\left({1\over T}+M_0\right)\ep\]
 and consequently, by  Lemma \ref{est:R2},
\be\label{est:curl:vb}\|\curl\intT\vb\|_{W^{k-4,\infty}(\mR^3)}\le CM_0\left({1\over T}+M_0\right)^{1\over2}\ep^{1\over2}.\ee

To ``remove'' the $\curl$ operator from \eqref{est:curl:vb},   we use $\vb=\cP\vb=-\curl\Delta^{-1}\curl\vb\implies(-\Delta)^{1\over2}\vb={\curl}( -\Delta)^{-{1\over2}}\curl\vb$ and the fact that $\curl(-\Delta)^{-{1\over2}} $ is a bounded mapping on any $W^{m,p}(\mR^3)$ space with $1<p<\infty$ (by properties of Fourier multipliers) to obtain
\[\|(-\Delta)^{1\over2}\intT\vb\|_{W^{m,p}(\mR^3)}\le C\|\curl\intT\vb\|_{W^{m,p}(\mR^3)} .\]
Then, apply the Hardy-Littlewood-Sobolev fractional integration theorem $\|g\|_{L^{3p\over3-p}}\le C\|(-\Delta)^{1\over2} g\|_{L^p}$ to the LHS and the interpolative H\"older's inequality $\|g\|_{L^{p}}\le\|g\|_{L^2}^{2\over p}\|g\|_{L^\infty}^{1-{2\over p}}$ to the RHS to arrive at, with $2<p<3$,
\[\ba&\quad\;\|\intT\vb\|_{W^{m,{3p\over3-p}}(\mR^3)}\\&\le C\|\curl\intT\vb\|_{W^{m,2}(\mR^3)}^{2\over p}\|\curl\intT\vb\|^{1-{2\over p}}_{W^{m,\infty}(\mR^3)}.\ea\]Finally, plug in \eqref{energy:MHD} and \eqref{est:curl:vb}, we complete the proof of \eqref{est:vb} by setting $m=k-4$, $1/p=1/3+1/s$.
\end{proof}

 \bigskip \section{\bf Appendices}\label{sec:app} 
\subsection{Geometric proof of Proposition \ref{thm:Navier}}
\begin{proof}
Throughout this proof,   let $\nm$ denote the outward normal at $\pa\Omega$ and let $\tg$ denote a typical tangent vector at $\pa\Omega$.

 By  identities \eqref{nabla:vu:ids}, we have,  \be\label{nm:tg:ids}\left\{\begin{aligned}\big((\grad\vu)\nm\big)\ccdot\tg&= (\nm\cn\vu)\ccdot\tg,\\\big((\grad\vu)\trsp\nm\big)\ccdot\tg&=( \tg\cn\vu)\ccdot\nm.\end{aligned}\right.  \ee
 
Next,  the assumption $\vu\ccdot\nm\bc=0$ implies $\tg\ccdot\grad(\vu\ccdot\nm)\bc=0$ so that by treating $\tg\ccdot\grad$ as a directional derivative and using the product rule, we have  at $\pa\Omega$, \[(\tg\cn\vu )\ccdot\nm=-(\tg\cn \nm) \ccdot\vu . \]  Combine it with \eqref{nm:tg:ids} to obtain, at a general smooth boundary $\pa\Omega$ with $\vu\ccdot\nm\bc=0$   \be\label{BC:equiv:0}\left\{\begin{aligned}\big((\grad\vu)\nm\big)\ccdot\tg&=\Big [{\nm\cn\vu}\Big]\ccdot\tg,\\\big((\grad\vu)\trsp\nm\big)\ccdot\tg&=-( \tg\cn\nm)\ccdot\vu.\end{aligned}.\right.  \ee

  Recall the definition $\SS=\grad\vu+(\grad\vu)\trsp$ and apply the above identities to obtain,
 \be\label{BC:equiv}(\SS\nm)\ccdot\tg=\Big [{\nm\cn\vu}\Big]\ccdot\tg-( \tg\cn\nm)\ccdot\vu.\ee
One can also write the above identity using   the vorticity. In fact, combine identities \eqref{BC:equiv:0}  with $(\curl\vu)\times\nm=(\nabla\vu-(\nabla\vu)\trsp)\nm$  to have  \[\big((\curl\vu)\times\nm\big)\ccdot\tg=\Big [{\nm\cn\vu}\Big]\ccdot\tg+( \tg\cn\nm)\ccdot\vu.\] Subtract it from \eqref{BC:equiv} to arrive at \be\label{Navier:curl}(\SS\nm)\ccdot\tg=\big((\curl\vu)\times\nm\big)\ccdot\tg-2( \tg\cn\nm)\ccdot\vu \ee

Now, regarding the $-\tg\cn\nm$ term, it is associated with the shape operator\footnote{http://mathworld.wolfram.com/ShapeOperator.html },
\[{\mathscr S}(\tg):=-\tg\cn \nm\]
which is a linear mapping in any given tangent plane of $\pa\Omega$. Then, the {\it symmetric} bilinear form 
\be\label{II}\II(\tg_1,\tg_2):={\mathscr S}(\tg_1)\ccdot\tg_2=-(\tg_1\cn\nm)\ccdot\tg_2,\ee
 defined for any two tangent vectors $\tg_1,\tg_2$ is the the second fundamental form\footnote{http://mathworld.wolfram.com/SecondFundamentalForm.html }. Its symmetry can be shown straightforward   e.g. by choosing the surface as a level set of scalar function $ g $, so that $\nm=\dfrac{\nabla  g }{|\nabla g |}$, which we will skip. Combining such symmetry with \eqref{BC:equiv}, \eqref{Navier:curl}  and the assumption $\vu\ccdot\nm\bc=0$, we prove \eqref{Navier:shapeO}.
 
To prove \eqref{shapeO:kappa}, we recall that the (orthonormal)  principal directions\footnote{https://en.wikipedia.org/wiki/Principal\_curvature} are  the two (orthonormal) eigenvectors of the shape operator 
\be\label{shapeO:eig}{\mathscr S}(\vce_i)=-\vce_i\cn \nm=\kappa_i\vce_i,\quad i=1,2,\ee
where eigenvalue $\kappa_i$ denotes the  principal curvature associated with $\vce_i$. They   effectively diagonalize $\II$, i.e. by definition \eqref{II}, $$ \II(\vce_i,\vce_j) =\begin{cases}\kappa_i,& i=j,\\
 0,&i\ne j.\end{cases}
$$  Combine \eqref{shapeO:eig},   \eqref{Navier:shapeO} to obtain  \eqref{shapeO:kappa}.

 Finally, for the special case of $\Omega$ being a spherical shell, any pair of orthonormal tangent vectors can serve as the  principal directions, and therefore \eqref{Navier:kappa:SH} easily follows from the fact that $\nm=\pm\ver$ at $r=1\pm\thick$ and from  \eqref{shapeO:eig} so that   $$\kappa_i =-(\vce_i\cn\nm )\cdot\vce_i=\mp{1\over r}\quad\text{ at }\;r=1\pm\thick.$$  
\end{proof}
\subsection{Energy and enstrophy estimates in a thin shell with Navier boundary conditions}

Let norm $\|\cdot\|$ stand for the $L^2(\Omega)$ norm and $\langle\cdot,\cdot\rangle$ for the $L^2(\Omega)$ inner product. 
\begin{proposition}\label{prop:energy:decrease}
Consider \eqref{NS:vec} subject to the Navier boundary conditions \eqref{BC:Navier} with $\lambda\ge0$ and $\vg \equiv{\bf0}$. Then, the energy $\|\vu\|(t)$ is decreasing with time.
\end{proposition}
 \begin{proof}
Take the $L^2(\Omega)$ inner product of $\vu$ and the first equation of \eqref{NS:vec}, noting the Coriolis term  is perpendicular to $\vu$,
\be\label{LHS:Navier-Stokes}\ba \mu\langle \Delta\vu,\vu\rangle &=\langle \pt\vu,\vu\rangle+\langle \vu\cn\vu,\vu\rangle+\langle \grad q,\vu\rangle\\
&={1\over2}\pt\|\vu\|^2+\int_\Omega\big({1\over2}\vu\cn|\vu|^2+\vu\cn q\big) \\
&={1\over2}\pt\|\vu\|^2
\ea\ee
where the last step is due to the Divergence Theorem, zero-flux boundary condition $\vu\ccdot\nm\bc=0$ and $\dv\vu=0$.

Now,
it is useful to derive a  version of the Green's formula adapted to the Navier boundary conditions. First, use $\dv\vu=0$ to rewrite $\Delta\vu=\dv\SS$ so that 
\[\ba\langle \Delta\vu,\vu\rangle&=\langle\dv\SS,\vu\rangle
   =-\langle\SS,\nabla\vu\rangle+\int_{\pa\Omega}(\SS\nm)\ccdot\vu.
\ea\]
The second term is non-positive. In fact,
  at $\pa\Omega$, the Navier boundary conditions \eqref{BC:Navier} imply that $\vu$ is perpendicular to $\nm$ while $(\SS\nm+\lambda\vu)$ is parellel to $\nm$ for $\vg \equiv0$. Therefore, with $\lambda\ge0$,
\be\label{Sn:vu:neg}\int_{\pa\Omega}(\SS\nm+\lambda\vu)\ccdot\vu=0\text{\;\; so\;\;}\int_{\pa\Omega}(\SS\nm )\ccdot\vu\le 0.\ee
For the $\langle\SS,\nabla\vu\rangle$ term, use the definition of  inner-products between matrices and the fact $\SS\trsp=\SS$ to obtain,
$\langle\SS,\nabla\vu\rangle=\langle\SS\trsp,(\nabla\vu)\trsp\rangle=\langle\SS,(\nabla\vu)\trsp\rangle
. $ Therefore,
\[\langle\SS,\nabla\vu\rangle={1\over2}\langle\SS,\nabla\vu+(\nabla\vu)\trsp\rangle={1\over2}\|\SS\|^2.\]
So,  combine the above 3 equations to arrive at \[\langle\Delta\vu,\vu\rangle\le-{1\over2}\|\SS \|^2.\] Together with \eqref{LHS:Navier-Stokes}, it implies
\be\label{L2:est}\pt\|\vu\|^2\le-{\mu}\|\SS\|^2.\ee
The proof is complete.
\end{proof} 

To obtain some estimates on the total enstrophy $\|\curl\vu\|$, it suffices to estimate $\|\grad\vu\|^2$.
Simply taking the time integral of   the above inequality \eqref{L2:est} will however not yield estimate on $\int_0^T\|\grad\vu\|^2dt$ because $\SS$ lacks the information on the anti-symmetric part of $\nabla\vu$ (which actually conincides with $\curl\vu$). The remedy is to employ  another version of     Green's formula
\be\label{Green:0}\ba \langle \Delta\vu,\vu\rangle&=- \|\nabla\vu\|^2+ \int_{\pa\Omega}({\nm\cn\vu})\ccdot\vu\\
&=- \|\nabla\vu\|^2+ \int_{\pa\Omega}({\nm\cn\vuh})\ccdot\vuh,\ea\ee
where the zero-flux boundary condition $\vu\ccdot\nm\bc=0$ was also used.
For the boundary term above, apply the first equation of \eqref{Navier:kappa:SH} and the fact that $\nm=\pm \ver $ at the boundaries $r=1\pm\delta$ to obtain, 
\[{\nm\cn\vuh}\Bc =[\SS \nm  ]_\ttan  + (\nm\ccdot\ver){ \vuh\over  r}.\]
Then,  \[\begin{split}\int_{\pa\Omega}({\nm\cn\vuh})\ccdot\vuh&=\int_{\pa\Omega}\Big([\SS \nm  ]_\ttan  + (\nm\ccdot\ver){ \vuh\over  r}\Big)\ccdot\vuh\\&=\int_{\pa\Omega}\Big([\SS \nm  ]  + (\nm\ccdot\ver){ \vu\over  r}\Big)\ccdot\vu\\&\le\int_{\pa\Omega} (\nm\ccdot\ver){ \vu\over  r} \ccdot\vu \quad\cdots\text{ by  \eqref{Sn:vu:neg}}.\end{split}\]
Effectively, there are no more derivatives of $\vu$ in the boundary integral (indeed this  formula works for general domain). Thus, by  applying the divergence theorem to the right side above to obtain,
 \[\ba\int_{\pa\Omega}({\nm\cn\vuh})\ccdot\vuh\le\,& \int_\Omega\dv(\ver|\vu|^2r^{-1})\\&\le {C }\left({\|\nabla\vu\|\|\vu\| }+{\|\vu\|^2 }\right)\ea\]
where  we also used the   H\"older's inequality.
Substitute it into \eqref{LHS:Navier-Stokes}, \eqref{Green:0} to obtain
\[\ba{1\over2}\pt\|\vu\|^2&\le\mu\Big(-\|\nabla\vu\|^2+ {C }\big({\|\nabla\vu\|\|\vu\| }+{\|\vu\|^2 }\big)\Big)\\&\le\mu\Big(-{1\over2}\|\nabla\vu\|^2+{ {C }\|\vu\|^2 }\Big),\ea\]with a different constant $C$.
Together with the decrease of energy $\|\vu\|^2\le \|\vu_0\|^2$ due to \eqref{L2:est}, it implies 
\be\label{grad:u:est} \int_0^T\|\nabla\vu\|^2\,dt\le \|\vu_0\|^2\Big({1\over\mu }+{{C }T }\Big).\ee

\end{document}